\documentclass[11pt]{article}
\usepackage[margin=1in]{geometry}
\usepackage[T1]{fontenc}
\usepackage{lmodern}
\usepackage{microtype}
\usepackage{amsfonts}
\usepackage{amsmath}
\usepackage{amsbsy}
\usepackage{amsxtra}
\usepackage{latexsym}
\usepackage{amssymb}
\usepackage{amsthm}
\usepackage{url}
\usepackage{cite}
\usepackage{enumerate}
\makeatletter
\g@addto@macro\thesection
\makeatother

\newtheorem{theorem}{Theorem}
\newtheorem*{theorem*}{Theorem}

\newtheorem{corollary}{Corollary}
\newtheorem{proposition}{Proposition}
\newtheorem{remark}{Remark}
\newtheorem{definition}{Definition}
\newtheorem{example}{Example}

\DeclareMathOperator{\intr}{int}

\DeclareMathOperator{\argmin}{argmin}

\newcommand{\lng}{\langle}
\newcommand{\rng}{\rangle}
\newcommand{\lf}{\left}
\newcommand{\rg}{\right}

\newcommand{\R}{\mathbb R}

\newcommand{\mc}{\mathcal}

\newcommand{\si}{\sigma}

\newcommand{\Conv}{\mathcal C}

\newcommand{\Kup}{\mathcal K}

\begin{document}

\title{Converses of Moreau's decomposition theorem}
\author{A. B. N\'emeth\\[3pt]
\small Faculty of Mathematics and Computer Science, Babe\c{s}--Bolyai University\\
\small Str. Kog\u{a}lniceanu nr. 1--3, RO-400084 Cluj-Napoca, Romania\\
\small \texttt{nemab@math.ubbcluj.ro}
\and
S. Z. N\'emeth\\[3pt]
\small School of Mathematics, University of Birmingham\\
\small Watson Building, Edgbaston, Birmingham B15 2TT, United Kingdom\\
\small \texttt{s.nemeth@bham.ac.uk}}
\date{}

\maketitle

\begin{abstract}
We study when generalized metric projections generated by strongly quasiconvex functions are polarizable. The point of departure is Moreau's
decomposition theorem, where the Euclidean metric projections onto two mutually polar closed convex cones form complementary retractions whose sum
is the identity. This suggests the following fundamental question: when does a given retraction admit such a complementary polar retraction? We
formulate this notion through polarizable retractions and then specialize it to generalized projections of the form 
$x\mapsto P_\Conv^\varphi(x)$, defined by
minimizing $\varphi(x-y)$ with respect to $y$ over a closed convex set $\Conv$. After recalling basic properties of mutually 
polar retractions and
first-order characterizations of generalized projections, we prove the main result: If $\varphi$ is strongly quasiconvex, the 
generalized projection $P_\Conv^\varphi$ is polarizable if and only if $\Conv$ is a closed convex cone. Thus, polarizability of the projection is
characterized by the conic geometry of its range. Hence, this result can be considered as a converse of Moreau's decomposition theorem. As a consequence, a more specific converse of the same theorem is presented for the Euclidean norm.
\end{abstract}

\medskip

\section{Introduction}

Duality is one of the basic principles of functional analysis and convex analysis. In an inner product space, duality has a particularly concrete
geometric form because the inner product allows orthogonality, polar cones, and nearest-point projections to interact directly. Metric projection
onto closed convex sets, and especially onto closed convex cones, therefore provides a natural setting in which geometric duality can be expressed analytically in terms of mappings of the ambient space.

A classical manifestation of this principle is Moreau's decomposition theorem. For a closed convex cone and its polar cone in a Hilbert space,
every vector decomposes into two orthogonal components, and these components are precisely the metric projections onto the two cones. In mapping
language, the two projections are retractions, their sum is the identity mapping, and each projection extracts the part of the vector belonging to
one additive term in the polar decomposition. This motivates the notion of a pair of mutually polar retractions.

The mapping viewpoint raises a natural question that is more general than Moreau's original setting. Given a retraction $Q$ of a vector space,
when is its complementary mapping $I-Q$ again a retraction? Equivalently, when can $Q$ be completed to a pair of mutually polar retractions whose
sum is the identity? We call such a retraction \emph{polarizable}. The purpose of this note is to answer this question for a class of generalized
metric projections. Instead of measuring displacement by the Euclidean norm, we consider a strongly quasiconvex, differentiable, nonnegative
function $\varphi$ with $\varphi(0)=0$, and define the $\varphi$-projection onto a closed convex set $\Conv$ by minimizing $\varphi(x-y)$
over $y\in\Conv$. The ordinary Euclidean metric projection is recovered as a special case, but the formulation permits a substantially broader family of projection-type mappings.

Our main result shows that, within this framework, polarizability is governed entirely by the geometry of the target set: the generalized
projection $P_\Conv^\varphi$ is polarizable if and only if $\Conv$ is a closed convex cone. Thus, although the 
projection is generated by the function $\varphi$, the existence of a complementary retraction forces $\Conv$ to be a closed convex cone, yielding
a geometrical structure which resembles Moreau's decomposition. In the particular case of Euclidean projection our main result yields a converse of Moreau's decomposition theorem. Before proving our results, we collect the 
required terminology, elementary facts about mutually polar retractions, and first-order optimality descriptions of
generalized projections.

\section{Terminology}

If $X$ is a real topological vector space, then a convex set $\Conv \subseteq X$ is called \emph{strictly convex} if
\[tu+(1-t)v\in\intr(\Conv),\] for any $u,v\in\Conv$ and any $t\in (0,1)$.
\medskip
	
If $X$ is a real vector space, then a set $\Kup\subseteq X$ is called \emph{convex cone} if it satisfies

\begin{enumerate}[(i)]
	\item $\lambda x\in \Kup$, for all $x\in \Kup$ , $\lambda \in \R_+:=\{t\in\mathbb R:t\ge 0\}$, 
	\item $x+y\in \Kup$, for all $x,y\in \Kup$, 
\end{enumerate}

If $X$ is a Hilbert space with inner product $\langle\cdot,\cdot\rangle$, then the \emph{polar cone} of a closed convex cone 
$\Kup\subseteq X$ is the closed convex cone defined by 

\[\Kup^{\circ}=\{x\in X:\langle x,y\rangle\le0,\,\forall y\in\Kup\}.\]

We have $\Kup^{\circ\circ}=\Kup$, where 
$\Kup^{\circ\circ}:=(\Kup^\circ)^\circ$. The closed convex cones $\Kup,\mathcal{L}:=\Kup^\circ$ are called \emph{mutually 
polar cones}.  

If $X =(X,\|\cdot\|)$ is a Banach space and $\Conv\subseteq X$ a nonempty closed convex cone,
then the mapping
$$P_\Conv x=\argmin \{\|x-y\|:y\in\Conv\}$$
is well defined and unique for each $x\in X$ and it is called \emph{the metric projection of $x$ onto $\Conv$.}

In the most of our consideration
$X= \R^n$ with $\R^n = (\R^n, \langle\cdot,\cdot\rangle)$ the Euclidean $n$-space generated by the orthonormal reference system $e_1,...,e_n$
and the inner product it defines.

A Riesz space is a real vector space $X$ equipped with a partial
order $\leq$ satisfying the compatibility conditions
\[
x \leq y \implies x+z \leq y+z
\qquad \text{for all } x,y,z\in X,
\]
and
\[
x \leq y,\quad \lambda\geq 0
\implies \lambda x\leq \lambda y,
\]
such that every pair $x,y\in X$ has a least upper bound $x\vee y$
and a greatest lower bound $x\wedge y$.

The positive cone of $X$ is
\[
\mathcal{K}:=\{x\in X:x\geq 0\}.
\]
The partial order $\le$ of $X$ is induced by the positive cone of $X$ as follows:
\[
x\leq y \iff y-x\in\mathcal{K}.
\]

A standard example of a Riesz space is $\mathbb{R}^n$ equipped
with the partial order induced by a simplicial cone
$\mathcal{K}\subseteq\mathbb{R}^n$, namely,
\[
x\leq y \iff y-x\in\mathcal{K}.
\]
Recall that a cone $\mathcal{K}\subseteq\mathbb{R}^n$ is called
simplicial if it is generated by $n$ linearly independent vectors
$v_1,\ldots,v_n\in\mathbb{R}^n$, that is,
\[
\mathcal{K}
=\left\{\sum_{i=1}^n\lambda_i v_i:
\lambda_i\geq 0,\ i=1,\ldots,n\right\}.
\]
In particular, the usual componentwise order corresponds to
$\mathcal{K}=\mathbb{R}^n_+$.

\section{Retractions: definition  and examples}\label{mutu}

In this section we suppose that $X$ is a real vector space.

The  mapping $Q:X\to X$ is called \emph{retraction} if  it
is \emph{idempotent}, that is, if $Q^2=Q$.

\begin{definition}\label{polab|}

The retraction $Q: X \to X$ is called \emph{polarizable} if  $I-Q:X \to X$ with $I$ the identity mapping of
$X$ is a retraction.

\end{definition}

\begin{definition}\label{muture} 
The mappings $P,R:\,X\to X$ are called
\emph{mutually polar retractions} if 
\begin{enumerate}[(i)]
	\item $P$ and $R$ are retractions, 
	\item $P+R=I$,
\end{enumerate}
\end{definition}

We will use the notations $PX=\mc P$ and $RX=\mc R$.

 We will say that $P$ and $R$ 
 are \emph{polars of each other}. Thus, $P$ and $R$ are polarizable.
 Also, the sets $\mc P$ and $\mc R$ are said {\em mutually polar}.

The retraction $R$ is polarizable if and only if $R$ and $I-R$ are mutually polar
retractions.

\begin{example}\label{triv}
The mappings $I$ and $0$ are mutually polar retractions.
Hence both are polarizable.
\end{example}
%\end{remark}

\begin{example}\label{ex0}
If $\R^n$ is the Euclidean $n$-space
endowed wih the orhonormal system $e^1,...,e^n$ and the reference system it defines, then putting 
$x=t_1e^1+\dots+t_ne^n$, consider the mappings
$$Qx=|t_1|e^1+\dots+|t_n|e^n$$
and
$$Px=t_1^+e^1+\dots+t_n^+e^n, $$
where $t^+$ denotes the positive part of the real number $t$.

$Q$ and $P$ are both retractions onto the positive orthant $\R^n_+$. 
$I-P$ is retraction too, but $I-Q$ is not. Hence $P$ is polarizable while  $Q$ is not.
\end{example}

\begin{example}\label{eex2}

If $\mc K\subseteq X$ is a closed convex cone in a Hilbert space $X$ with inner product 
$\lng\cdot,\cdot\rng$ and $\mc K^\circ$ is the polar of  $\mc K$,
then  the metric projections $P_\mc K$ and $P_{\mc K^\circ}$ are mutually polar retractions.
This is the consequence of the following decomposition theorem of Moreau (see
\cite{Moreau1962}, \cite{Zarantonello1971}). 

\begin{theorem} [Moreau]\label{moreau}
	Let  $\mc K,\mc L\subseteq X$ be two mutually polar closed 
	convex cones, where $X$ is a Hilbert space with inner product $\lng\cdot,\cdot\rng$. 
	Then, the following statements are equivalent:
	\begin{enumerate}
		\item[(i)] $z=x+y,~x\in \mc K,\,\,~y\in \mc L$ and $\lng x,y\rng=0$,
		\item[(ii)] $x=P_\mc K z$ and $y=P_\mc L z$,
	\end{enumerate}
	where $P_\mc K$ and $P_\mc L$ are the metric projections onto $\mc K$ and $\mc L$, 
	respectively.
\end{theorem}

\end{example}
 
\begin{example}\label{riesz}

If $(X,\mc K)$ is a Riesz space with positive cone $\mc K$ then $P$ and $R$ defined by
$Px=x^+$ and $Rx=-x^-$ are mutually polar retractions. Here $\mc P=\mc K$ and $\mc R = -\mc K$.

\end{example}

%%%%%%%%%%%%%%%%%%%%%%%%%%%%%%%%%%%%%%%%%%%%%%%%%%%%%%%%%%%%%%%%%%%%%%%

 \section{Properties of the mutually polar retractions}

The following proposition is a slight modification of Lemma 4.1 in \cite{NN}. We give a proof because the result is used below.

\begin{proposition}\label{kov}
	If $X$ is a vector space and $P,R:X\to X$ are mutually polar retractions, then
\begin{enumerate}[(i)]
	\item $PR=RP=0.$
	\item $P0=R0 =0$;
	\item  $\mc P \cap \mc R = \{0\}$ and $\mc P + \mc R= X$,
	\item $Px-Rx= 0$ if and only if $x=0$.
\end{enumerate}
    \end{proposition}

\begin{proof}
	$\,$

	\begin{enumerate}[(i)]
		\item From Definition \ref{muture}, it follows that
			\(PR=(P+R)R-R^2=R-R^2=0.\) Hence, by swapping $P$ and
			$R$, we also have $RP=0$.
		\item Suppose that $P0=x$. Then, by using item (ii) of Definition \ref{muture}, we obtain
			\begin{eqnarray*}
				x=P0=P^20=Px\implies x=Px+Rx=x+Rx\implies \\ 
				Rx=0
				\implies Rx=R^2x = R0\implies 
			P0=x=Px+Rx=P0+R0 = 0\implies \\R0=0.
			\end{eqnarray*}
		\item Let any $z\in\mc P\cap\mc R$. Hence, there exists
			$x,y\in X$ such that $z=Px=Ry$, which, by using 
			item (i), implies $z=Px=P^2x=PRy=0$. On the other hand, 
			$0\in\mc P\cap\mc R$ because of item (ii). The second
			equality follows from item (ii) of Definition \ref{muture}.
		\item The equality $P0-R0=0$ comes from item (ii). Suppose that $Px-Rx=0$. 
			Hence, by using the first equality of item (iii), we
			obtain $Px=Rx\in\mc P\cap\mc R=\{0\}$. Therefore, $x=Px-Rx=0$. 
	\end{enumerate}
\end{proof}
%%%%%%%%%%%%%%%%%%%%%%%%%%%%%%%%%%%%%%%%%%%%%%%%%%%%%%%%%
%%%%%%%%%%%%%%%%%%%%%%%%%%%%%%%%%%%%%%%%%%%%%%%%%%%%%%%%%%%%%%%%%%%%%%%%%%%%%%%%%%%

\section{The generalized metric projection:}

 A continuous function   $\varphi:\R^n \to \R$  is called \emph{ strongly quasiconvex}  if for each $\lambda \in (0,1)$
and each $x, y \in \R^n $ with  $x\neq y$ we have 
$$
\varphi ( \lambda x+(1-\lambda)y) < \max \{ \varphi (x), \varphi (y)\}.
$$
 
It follows from the definition that the function $\varphi$ is  strongly quasiconvex if and only if all of its  
nonempty sublevel sets $\{x\in \R^n : \varphi (x)\leq a \}$ for $a \in \R$,  are  
strictly convex (remember that $ \varphi $ is continuous). Moreover, for an arbitrary
nonempty convex set $\Conv$ ,
for each $x\in \R^n$ the  function  $\varphi_x(y) = \varphi(x-y)$ has only one minimizer in $\Conv$,
 see Theorem~\(3.5.9\) of \cite{Baz2006}.

\begin{definition}\label{foodef1}

Suppose that $\varphi$ is a strongly quasiconvex differentiable
nonnegative function with $\varphi(0)=0.$
Suppose $\Conv \subseteq \R^n$ is a nonempty closed convex set.
Then, for any  $x\in \R^n$ there is a well defined
mapping $P_\Conv^\varphi:\R^n\to \Conv$ called \emph{the $\varphi$-projection of $x$ onto $\Conv$}
given by the formula:
\begin{equation}\label{fookepl}
	P_\Conv^\varphi x=\argmin \{\varphi(x-y): y\in \Conv\}.
\end{equation}
%Hence, $P$ depends on $\varphi$ and $\Conv $ and if necessary 
%we will also use the notation $P_{\Conv}^{\varphi}.$

The \emph{distance of $x$ from $\Conv$} is defined by
\begin{equation}\label{foodef2} 
d(x,\Conv)= \varphi(x-Px).
\end{equation}

\end{definition}

\begin{remark}\label{remm1}
It is immediate from the formula (\ref{fookepl}) that the same generalized projection
is defined by many functions. For instance if $f: \R_+ \to \R_+$ is a continuous,
increasing, strongly quasiconvex function with $f(0)=0$,
then $P_\Conv^{\varphi} = P_\Conv^{f\circ \varphi}$.
\end{remark}

\begin{remark}\label{philosoph}

The above defined generalized projection is a natural 
generalization of the Euclidean metric projection
since this is precisely the projection
with respect to \[\si(x) :=\|x\|= \lf(x_1^2+...+x_n^2\rg)^{\frac 12}.\] 

\end{remark}

The next result is Corollary 2 of \cite{FN}.

\begin{theorem}\label{c2}
	Let $\Conv\subseteq \R^n$ be a nonempty closed convex set, $x\in\mathbb R^n$ and $p\in\Conv$. If $\varphi$ is 
	differentiable in $x-p$, then the following statements are equivalent:
\begin{enumerate}
	\item[(i)] $p=P^\varphi_\Conv x $; 
	\item[(ii)] $\langle\nabla\varphi(x-p),q-p\rangle\le 0$, for all $q\in\Conv$. 
\end{enumerate}
\end{theorem}

For convex cones, Theorem 5 of \cite{FN} gives the following characterization.

\begin{theorem}\label{t5}
	Let \(\Kup \subseteq \mathbb{R}^n\) be a closed convex cone,
	\(x \in \mathbb{R}^n\), and \(p \in\Kup\). Suppose that
	\(\varphi\) is differentiable at \(x-p\). Then
	\(
	p=P_\Kup^\varphi x 
	\)
	if and only if
	\[
	\nabla\varphi(x-p) \in \Kup^\circ
	\qquad\text{and}\qquad
	\langle \nabla\varphi(x-p),p\rangle = 0.
	\]
\end{theorem}

%%%%%%%%%%%%%%%%%%%%%%%%%%%%%%%%%%%%%%%%%%%%%%%%%%%%%%%%%%%%%%%%%%%%%%%%%%%%%%%%%%%%%%%%%%%%%%%%%%%%%%%%%%%%%%%

\subsection{The polarisability of the metric projection}

We now present our main result. This result can be considered as a kind of general 
converse of Moreau's decomposition Theorem. Although its sufficiency part was already established in \cite{NN}, we include the entire
proof here for completeness. 

\begin{theorem}\label{footet}

If $\varphi:\mathbb R^n\to\mathbb R^n$ is a strongly  quasiconvex, differentiable and nonnegative function with 
$\varphi(0)=0$, then the $\varphi$-projection $P^{\varphi}_\Conv$ onto the closed convex set $C\subseteq\R^n$ 
is polarizable if and only if $\Conv\subseteq\mathbb R^n$ is a closed convex cone.

\end{theorem}

\begin{proof}
	$\,$ 
	\medskip

%We begin by emphasising the relation:
%
%$$
% (I-P^{\varphi}_\Conv)^2 = I-P^{\varphi}_\Conv\iff P^{\varphi}_\Conv(I-P^{\varphi}_\Conv) = 0 .
%$$

\textbf{The necessity:}
	Suppose that $\varphi$ is strongly quasiconvex and let $\Conv$ be a closed 
	convex set such $P_\Conv^\varphi$ is polarizable. Then, we need to show that $\Conv$ is a 
	closed convex cone.

	From Proposition \ref{kov} (i) we get
	\begin{equation}\label{ne}
		P_\Conv^\varphi(x - P_\Conv^\varphi x) = 0,
	\end{equation} 
	for all $x \in \mathbb{R}^n$. 

	Hence, equation \eqref{ne} and Theorem \ref{c2} with
	$x - P_\Conv^\varphi x$ in place of $x$ yields

	\begin{equation*}
		\langle \nabla\varphi(x - P_\Conv^\varphi x), z\rangle
		=\langle \nabla\varphi(x - P_\Conv^\varphi x-P_\Conv^\varphi(x - P_\Conv^\varphi x)), 
		z-P_\Conv^\varphi(x-P_\Conv^\varphi x) \rangle \leq 0,
	\end{equation*} 
	\quad \text{for all } $z \in \Conv$.
	Hence, $\nabla\varphi(x - P_\Conv^\varphi x)\in\Kup^\circ$, where $\Kup$ is the cone generated by $\Conv$ (that 
		is, the smallest closed convex cone $\Kup$ containing $\Conv$ which is 
	$\Kup=\{\lambda v:v\in \Conv,\lambda\ge 0\}$).
Taking $z = P_\Conv^\varphi x$, we get

\begin{equation}\label{ineq1} 
	\langle \nabla\varphi(x - P_\Conv^\varphi x), P_\Conv^\varphi x \rangle \leq 0.
\end{equation}

By using Theorem \ref{c2} again, we obtain

\begin{equation}\label{ineq2} 
 \langle \nabla\varphi(x - P_\Conv^\varphi x),z-P_\Conv^\varphi x \rangle \leq 0,
\end{equation}
for any $z\in \Conv$. From equation \eqref{ne} we have $0\in \Conv$. Letting $z=0$ in inequality \eqref{ineq2},
we obtain

\begin{equation}\label{ineq3}
	\langle \nabla\varphi(x - P_\Conv^\varphi x),- P_\Conv^\varphi x \rangle\le 0,
\end{equation}

Thus, from equations \eqref{ineq1} and \eqref{ineq3}, it follows that 
\begin{equation*}\label{ineq4}
	\langle \nabla\varphi(x - P_\Conv^\varphi x), P_\Conv^\varphi x \rangle = 0,
\end{equation*}
where $P_\Conv^\varphi x \in \Conv \subseteq \Kup$, $\nabla\varphi(x - P_\Conv^\varphi x) \in \Kup^{\circ}$. Hence, 
Theorem \ref{t5} yields

\[P_\Conv^\varphi x = P_\Kup^\varphi x.\]

Therefore, $\Conv = \Kup$.
\medskip

\textbf{The sufficiency:}
Since \[(I-P^\varphi_\Conv)(I-P^\varphi_\Conv)=I-P^\varphi_\Conv-P^\varphi_\Conv(I-P^\varphi_\Conv),\] for 
$P^{\varphi}_ \Conv $ with $\Conv$ a closed convex 
cone, we have to show
\begin{equation}\label{foo1}
	P^{\varphi}_\Conv(I-P^{\varphi}_\Conv)=0.%, \quad (I-P^{\varphi}_\Conv)(\R^n)=\ker(P^{\varphi}_\Conv):=\{x\in\R^n:P^{\varphi}_\Conv(x)=0\}.
\end{equation}

Since $\Conv$ is closed under the addition, $P^{\varphi}_\Conv x+v\in \Conv$ for
any $v\in \Conv$ and since $$P^{\varphi}_\Conv x=\argmin\{\varphi (x-u):u\in \Conv\},$$
we have $$ \varphi(x-P^{\varphi}_\Conv x)\leq \varphi(x-P^{\varphi}_\Conv x-v)$$
and since the minimum point of $\varphi$ is unique, it  follows that
$$\argmin \{\varphi(x-P^{\varphi}_\Conv x-v): v\in \Conv\}=0.$$

Hence,
$$ P^{\varphi}_\Conv (x-P^{\varphi}_\Conv x) =\argmin\{\varphi(x-P^{\varphi}_\Conv x-v): v\in \Conv\}=0,$$
according to the above consideration. 

Since $x \in \R^n$ was arbitrarily chosen, we obtain %the first relation in 
\eqref{foo1}.

 The second equality of \eqref{foo1} easily
follows from its first equality. Thus the sufficiency is proved.
\end{proof}

The next corollary is a more explicit converse of Moreau's decomposition theorem and by using a slight adaptation of the ideas in the 
proof of Theorem \ref{footet} it can also be proved in Hilbert spaces. This corollary is an extension of Soltan's Theorem 1.2 in
\cite{Soltan2019}, where $Q$ was also supposed to be a metric projection onto a closed convex set.

\begin{corollary}[\bf Converse of Moreau's decomposition Theorem]
	Let $Q:\R^n\to R^n$ be a retraction, ${\cal C}\subseteq\R^n$ be a closed convex set 
	and $P_{\cal C}$ be the projection onto ${\cal C}$ 
	with respect to the Euclidean norm. If \[x=P_{\cal C} x+Qx, 
	\qquad\forall x\in\R^n,\] 
	then $\mc C$ is a closed convex cone and $Q=P_{\cal D}$, where ${\cal D}$ is the
	polar cone of $\mc C$.
\end{corollary}

\begin{proof}
	By Definition \ref{muture}, the maps $P_{\cal C},Q$ are mutually polar retractions. Hence, $P_{\cal C}$ is polarizable. Thus, Theorem \ref{footet} implies that 
	${\cal C}$ is a closed convex cone. Let ${\cal D}$ be the polar cone of $C$.
	The equality $P_{\cal C} x+Qx=x$ and Moreau's decomposition Theorem (i.e., Theorem \ref{moreau}) shows 
	that 
	\[Qx=x-P_{\cal C} x=P_{\cal D}x,\] for all $x\in\R^n$. 
\end{proof}

%%%%%%%%%%%%%%%%%%%%%%%%%%%%%%%%%%%%%%%%%%%%%%%%%%%%%%%%%%%%%%%%%%%%%%%%%%%%%%%%%
\section{Final Remarks}

Our main result identifies when a generalized projection $P_{\Conv}^\varphi$ onto a closed convex set 
$\Conv\subseteq\R^n$ with respect to a strongly quasiconvex function $\varphi:\R^n\to\R^n$ has a complementary 
retraction $I-P_{\Conv}^\varphi$: this occurs exactly when the set $\Conv$ is a closed convex cone. In particular, the 
following converse of Moreau's decomposition theorem holds: If two Euclidean projections onto closed convex sets are 
mutually polar retractions, then the two sets are mutually polar closed convex cones.

These conclusions place the main result close in spirit to Moreau's decomposition theorem while also separating what is 
specific to the Euclidean norm and what survives from it for a broader strongly quasiconvex generating function:
Moreau's decomposition theorem identifies the complementary retraction explicitly as the metric projection onto the polar 
cone. In the generalized setting
considered here, the central issue is the existence of the complementary retraction $I-P_\Conv^\varphi$; we
prove that this existence is equivalent to the set $\Conv$ being a closed convex cone, however in general $I-P_\Conv^\varphi$ is not a metric
projection. Still, this can be considered as a general converse of Moreau's decomposition theorem. In the case of the 
Euclidean norm the converse becomes more explicitly related to projections and the corresponding decomposition.

There is also a related rigidity phenomenon in the Banach-space setting. As discussed in \cite{Nemeth2025}, for the
ordinary metric projection $P_\Conv^\varphi$ onto an arbitrary closed convex cone $\Conv$, convexity properties of the 
range of $I-P_\Conv^\varphi$ are
closely tied to the underlying norm being generated by an inner product. This indicates that complementary projection mappings can encode not only properties of the projected set but also geometric information about the ambient space.

It remains of interest to determine whether differentiability of \(\varphi\) can be relaxed and when the complementary retraction admits a geometric description analogous to projection onto a polar cone.

%%%%%%%%%%%%%%%%%%%%%%%%%%%%%%%%%%%%%%%%%%%%%%%%%%%%%%%%%%%%%%%%%%%%%%%%%%%%%%%%%%%%%%%%%%%%%%%%%%%%%%%
\section{Declaration about use of AI}

During the finalisation of this paper, the authors used ChatGPT Plus as an editorial tool. In particular, it was used
to help refine the abstract, the introduction, and the final remarks; and to assist with language and presentation. The mathematical statements, 
arguments, proofs, and conclusions remained under the authors' control, and no new mathematical result was introduced on the
basis of the AI-generated text. All AI-assisted changes were reviewed or updated by the 
authors, who take full responsibility for the content of the paper.

%%%%%%%%%%%%%%%%%%%%%%%%%%%%%%%%%%%%%%%%%%%%%%%%%%%%%%%%%%%%%%%%%%%%%%%%%%%%%%%%%%%%%%%%%%%%%%%%%%%%%%%

\end{document}